\documentclass[a4paper, 12pt, twoside, notitlepage,leqno]{amsart}
\usepackage{amsmath,amscd}
\usepackage{amssymb}
\usepackage{amsthm}
\usepackage{enumerate}
\usepackage{fullpage}

\usepackage{mathrsfs}  
\usepackage{enumitem}
\usepackage{orcidlink}

\usepackage[utf8]{inputenc}
\usepackage{csquotes}
\usepackage{verbatim}
\usepackage[T1]{fontenc}
\usepackage{textcomp}  
\usepackage{mathtools,amssymb}
\allowdisplaybreaks

\newtheorem{theorem}{Theorem}[section]
\newtheorem{lemma}[theorem]{Lemma}
\newtheorem{corollary}[theorem]{Corollary}

\theoremstyle{definition}
\newtheorem{definition}[theorem]{Definition}

\newtheorem{proposition}[theorem]{Proposition}

\theoremstyle{remark}
\newtheorem{remark}[theorem]{Remark}

\numberwithin{equation}{section}

\newcommand{\R}{\mathbb{R}}

\newcommand{\WF}{\mathrm{WF}_{\mathrm{a}}}

\newcommand{\Beq}{\begin{equation}}
\newcommand{\Eeq}{\end{equation}}
\newcommand{\beq}{\begin{equation*}}
	\newcommand{\eeq}{\end{equation*}}
\newcommand{\bal}{\begin{align}}
	\newcommand{\eal}{\end{align}}

\usepackage{mathtools}

\newcommand{\bp}{\begin{prob}}
	\newcommand{\ep}{\end{prob}}
\newcommand{\bpr}{\begin{proof}}
	\newcommand{\epr}{\end{proof}}

\newcommand{\bel}[1]{\begin{equation}\label{#1}}
\newcommand{\ee}{\end{equation}}

\newcommand{\bsf}{\boldsymbol{f}}

\usepackage{scalerel,stackengine}
\stackMath
\newcommand\reallywidehat[1]{%
\savestack{\tmpbox}{\stretchto{%
  \scaleto{%
    \scalerel*[\widthof{\ensuremath{#1}}]{\kern-.6pt\bigwedge\kern-.6pt}%
    {\rule[-\textheight/2]{1ex}{\textheight}}%WIDTH-LIMITED BIG WEDGE
  }{\textheight}% 
}{0.5ex}}%
\stackon[1pt]{#1}{\tmpbox}%
}
\begin{document}

\title[Matrix Weighted double fibration transforms]{The matrix weighted real-analytic double fibration transforms}

\author[H. Chihara]{Hiroyuki Chihara}
\address{College of Education, University of the Ryukyus, Nishihara, Okinawa 903-0213, Japan}
\email{hc@trevally.net}

\author[S. R. Jathar]{Shubham R. Jathar}

\address{Computational Engineering, School of Engineering Sciences,
Lappeenranta-Lahti University of Technology LUT, Lappeenranta, Finland}

\email{shubham.jathar@lut.fi}

\author[J. Railo]{Jesse Railo}
\address{ Computational Engineering, School of Engineering Sciences,
Lappeenranta-Lahti University of Technology LUT, Lappeenranta, Finland \& Department of Mathematics, Stanford University, Stanford, CA, USA}
\email{jesse.railo@lut.fi; railo@stanford.edu}

%    General info


\begin{abstract} We show that the real-analytic matrix-weighted double fibration transform determines the analytic wavefront set of a vector-valued function. We apply this result to show that the matrix weighted ray transform is injective on a two-dimensional, non-trapping, real-analytic Riemannian manifold with strictly convex boundary. Additionally, we show that a real-analytic Higgs field can be uniquely determined from the nonabelian ray transform on real-analytic Riemannian manifolds of any dimension with a strictly convex boundary point.

\end{abstract}

\maketitle
\section{Introduction}\label{sec:Intro}

Let $(M,g)$ be a compact, oriented Riemannian manifold of dimension $n \geq 2$ with strictly convex boundary, and let $SM$ denote its unit tangent bundle. For a smooth function $f \in C^\infty(M)$, the \emph{geodesic X-ray transform} is defined by
\[
R f (z) = \int_{0}^{\tau_+(z)} f\big(x_z(t)\big)\, \mathrm{d}t, \qquad z \in \partial_+ SM,
\]
where $x_z$ is the unit-speed geodesic determined by the initial data $z$, and $\tau_+(z)$ is the forward travel time of $x_z$. In the Euclidean case, this reduces to the classical X-ray transform, and in two dimensions, it coincides with the Radon transform \cite{radon_uber_1917}. A matrix weighted generalization of the geodesic X-ray transform is given by
\begin{equation}
    R_W \bsf(z) = \int_{0}^{\tau_+(z)} W\big(x_z(t), \dot{x}_z(t)\big) \bsf\big(x_z(t)\big)\, \mathrm{d}t,
\end{equation}
where $\bsf \in C^\infty(M, \mathbb{C}^N)$ and $W \in C^\infty(SM, \mathrm{GL}(N, \mathbb{C}))$ is a smooth matrix weight. 

A notable special case is the \emph{attenuated ray transform}, where the weight arises from a matrix-valued attenuation $\mathcal{A} \in C^\infty(SM, \mathbb{C}^{N \times N})$. The associated weight $W_{\mathcal{A}}$ is defined as the solution to the transport equation
\begin{align}
    \begin{split}
        \frac{\mathrm{d}}{\mathrm{d}t} W_{\mathcal{A}}\big(x_z(t), \dot{x}_z(t)\big) + \mathcal{A}\big(x_z(t), \dot{x}_z(t)\big) W_{\mathcal{A}}\big(x_z(t), \dot{x}_z(t)\big) &= 0, \\
      W_{\mathcal{A}}|_{\partial_+ SM} &= \mathrm{Id}.
    \end{split}
\end{align}
The attenuated ray transform with attenuation $\mathcal{A}$ is then defined as $R_{W_{\mathcal{A}}^{-1}}$. The associated inverse problem asks whether one can recover $f$ from its (weighted) integrals over a family of geodesics.  For a recent survey on inverse problems for ray transforms, including weighted and general flow cases, we refer to \cite{Zhou:2024:survey, jathar2024twisted_survey}.

Recently, \cite{mazzucchelli2023general} studied the weighted geodesic ray transform and related transforms on real-analytic manifolds using the framework of double fibration transforms. In this article, we consider the matrix weighted version of these problems and their applications. First, we show that both the matrix weighted ray transform and the attenuated ray transform can be viewed as special cases of the matrix weighted double fibration ray transform. We prove that the real-analytic matrix weighted double fibration transform determines the analytic wavefront set of a vector-valued function. This result allows us to establish local injectivity for the matrix weighted ray transform and the attenuated ray transform. Using a layer stripping argument, we further obtain global injectivity results under the assumption that the manifold is real-analytic, two-dimensional, has strictly convex boundary, and non-trapping geodesic flow. As an application, we show that an arbitrary real-analytic Higgs field can be recovered from the corresponding nonabelian ray transform on a real-analytic manifold of dimension $n \geq 2$. Notably, for this result, we only assume that the boundary is strictly convex at a single point.
\subsection{Related literature}
These questions may be generalized to the case where $f$ is replaced by a symmetric tensor field, known as \emph{tensor tomography}. A particularly tractable setting is that of a \emph{simple manifold}, characterized by non-trapping geodesic flow, absence of conjugate points, and strictly convex boundary. Injectivity results for the geodesic ray transform on simple surfaces include, for example, functions \cite{Muhometov:1977}, symmetric $2$-tensor fields \cite{Sharafutdinov:2007}, and symmetric tensor fields of arbitrary order \cite{Paternain_Salo_Uhlmann2013}. Using analytic microlocal analysis, support theorems have been proved for the geodesic ray transform on real-analytic simple manifolds: for functions in \cite{Krishnan:2009}, and for symmetric $2$-tensor fields in \cite{Krishnan:Stefanov:2009}, with a remark that the argument extends to arbitrary tensor fields with minor modifications.

The study of matrix weighted transforms dates back at least to \cite{Vertgeu:1991}, using techniques from \cite{Muhometov:1977}. For scalar attenuations on simple surfaces, the attenuated ray transform is injective and admits stable inversion \cite{Salo_Uhlmann:2011}. In a similar setting, \cite{Monard:2016} establishes Fredholm-type inversion formulas for the attenuated geodesic X-ray transform acting on either functions or vector fields. Furthermore, \cite{Assylbekov:Monard:Uhlmann:2018} characterizes the full range and provides exact reconstruction formulas for the transform applied to the sum of a function and a one-form. For real-analytic simple manifolds, injectivity and stability results for scalar attenuations acting on pairs consisting of symmetric $2$-tensor fields and vector fields are obtained in \cite{Assylbekov:2020}. Using analytic microlocal methods, \cite{Homan:Zhou:2017} proves injectivity and stability for a generic class of generalized Radon transforms that integrate over analytic hypersurfaces on compact analytic Riemannian manifolds with boundary, under the Bolker condition.

It is conjectured in \cite{Paternain_Salo_Uhlmann2013} that the geodesic ray transform is injective on simply connected, strictly convex surfaces with non-trapping geodesic flow. In higher dimensions ($n \geq 3$), injectivity for symmetric $m$-tensor fields with $m \geq 2$ remains open on simple manifolds. The attenuated ray transform is known to be unstable in the presence of conjugate points on surfaces \cite{fold_caustic, Holman:Monard:Stefanov:2018}. On manifolds admitting a strictly convex foliation, global injectivity and stability for $R_W$ in dimensions $n \geq 3$ follow from scattering calculus arguments \cite{Paternain_Salo_Uhl_Zhou:2019}, developed in \cite{Uhlmann:Vasy:2016, Stefanov:Uhlmann:Vasy:2018}. Related developments include \cite{Vasy:Zachos:2024, Vasy:2024, Rafe:Monard:2024}. Injectivity for piecewise constant functions on non-trapping manifolds was established in \cite{Ilmavirta:Lehtonen:Salo:2020}, and extended to the case of matrix weights in \cite{Ilmavirta:Railo:2020}. For further results, see \cite{Todd:1980, Novikov:2002:inv, Boman:1993, Bohr_Paternain:2023, Jathar:loop}.

The nonlinear inverse problem of the injectivity of the \emph{nonabelian ray transform} is related to the attenuated ray transform via pseudo-linearization. This transform has applications in polarimetric neutron tomography \cite{Desai_et_all:2020, hilger2018tensorial}, soliton theory \cite{Manakov_Zakharov:1981, Ward:1988}, and coherent quantum tomography \cite{Ilmavirta:2016}; see also \cite{Novikov:2019}. In \cite{Novikov:2002}, reconstruction of a connection up to gauge in $\mathbb{R}^n$ ($n \geq 3$) is shown, along with a counterexample for $n = 2$. Under compact support assumptions, injectivity up to gauge in $\mathbb{R}^2$ is proved in \cite{Eskin:2004}. 

For simple manifolds, injectivity with Hermitian connections is shown in \cite{Sharafut:2000}, and local uniqueness under small $C^1$ norm assumptions is given in \cite{Finch_Uhlmann:2001}. Injectivity up to natural gauge for attenuations given by a unitary connection plus a skew-Hermitian Higgs field on simple surfaces is proved in \cite{Paternain:Salo:Uhlmann:2012}, and extended to general connections and Higgs fields in \cite{Paternain:Salo:2020:a}. On negatively curved manifolds of any dimension, injectivity is shown in \cite{Guillarmou:Paternain:Salo:Uhlmann:2016}. For real-analytic simple manifolds, generic injectivity holds for arbitrary real-analytic invertible weights \cite{Zhou:2017}. Using Carleman estimates, injectivity for general connections and Higgs fields on negatively curved manifolds is established in \cite{paternain2021carleman}. The stability and statistical analysis of the nonabelian ray transform are studied on simple surfaces \cite{Monard:Nickl:Paternain:2021} and in higher dimensions \cite{bohr2021stability,Bohr:Nickl:2024}. For recent related results, see for instance \cite{Chihara:2026,busch2025_1,busch2025}.

\subsection{Main results}\label{subsec:main}
We briefly recall the mathematical framework of our work. Further details are provided in the subsequent sections. Let $\mathcal{G}$ and $X$ be smooth, oriented manifolds of dimensions $m$ and $n$, respectively. Let $Z$ be an oriented, embedded submanifold of $\mathcal{G} \times X$, with projections $\pi_{\mathcal{G}}: Z \rightarrow \mathcal{G}$ and $\pi_X: Z \rightarrow X$. We assume that $m \geq n$ and
\[
\dim(Z) = m + n' \quad \text{for some } 1 \leq n' \leq n-1,
\]
and that both $\pi_{\mathcal{G}}$ and $\pi_X$ are \emph{submersions}. Such $Z$ is called a \emph{double fibration}. For each $z \in \mathcal{G}$, the fiber $\pi_\mathcal{G}^{-1}(z)$ and its image in $X$ are both $n'$-dimensional. We define
\begin{equation}\label{eq:G_z}
    G_z := \pi_X(\pi_\mathcal{G}^{-1}(z)) \subset X.
\end{equation}
Likewise, for each $x \in X$, the fiber $\pi_X^{-1}(x)$ is $(m - n'')$-dimensional, and its image in $\mathcal{G}$ is
\[
H_x := \pi_\mathcal{G}(\pi_X^{-1}(x)) \subset \mathcal{G},
\]
where $n'' := n - n'$. Throughout, we use the notation
\[
\mathrm{GL}(N, \mathbb{C}) := \left\{\, B \in \mathbb{C}^{N \times N} \mid \det(B) \neq 0 \,\right\}
\]
to denote the general linear group of invertible \(N \times N\) complex matrices.

We generalize \cite[Definition~1.5]{mazzucchelli2023general} to include matrix weights instead of scalar weights. 
\begin{definition}\label{def:double_fibration}
Let $\omega_{G_z}$ be the induced orientation form on $G_z$. Let $W \in C^\infty(\mathcal{G} \times X, \mathrm{GL}(N,\mathbb{C}))$ be a smooth, matrix-valued function. The linear map $R_W: C_c^{\infty}(X,\mathbb{C}^N) \rightarrow C^{\infty}(\mathcal{G},\mathbb{C}^N)$ defined by
\begin{equation}\label{eq:R_W_def}
    R_W \bsf(z) = \int_{G_z} W(z, x) \bsf(x) \, \mathrm{d}\omega_{G_z}(x), \quad z \in \mathcal{G}
\end{equation}
is called a \emph{matrix weighted double fibration transform}. We say that $R_W$ is \emph{real-analytic} if all related objects 
($\mathcal{G}, X, Z, \omega_{G_z}, W$) are real-analytic.
\end{definition}

Let $\Sigma \subset X$ be a smooth submanifold. The \emph{conormal bundle} of $\Sigma$, denoted by $N^{*}\Sigma$, is defined as
\[
N^{*}\Sigma = \left\{\,(x,\xi)\in T^{*}X \setminus 0 \,\middle|\, x\in\Sigma,\; \xi(v)=0 \text{ for all } v\in T_{x}\Sigma \, \right\}.
\]
Similar to~\cite[Theorem~2.2]{mazzucchelli2023general}, the delta distribution supported on $Z$ is an oscillatory integral, implying that $R_W$ is a \emph{Fourier integral operator} with the canonical relation
\begin{align}\label{eq:twisted_conormal}
    C :=(N^* Z \backslash\{0\})'= \left\{\,(z,\zeta,x,\eta) \,\middle|\, (z,\zeta,x,-\eta) \in N^*Z \setminus \{0\}\, \right\}.
\end{align}
We define the \emph{left} and \emph{right} projections
\[
\pi_L \colon C \to T^*\mathcal{G} \setminus 0, \qquad 
\pi_R \colon C \to T^*X \setminus 0.
\]
We say that the \emph{Bolker condition} holds at $(z,\zeta,x,\eta)\in C$ if $\pi_L^{-1}(z,\zeta) = \{(z,\zeta,x,\eta)\}$ and $\mathrm{d}\pi_L|_{(z,\zeta,x,\eta)}$ is injective.

The following theorem is a generalization of \cite[Theorem~1.2]{mazzucchelli2023general} to the matrix weighted case.

\begin{theorem}\label{tm:main}
Let $R_W$ be a real-analytic matrix weighted double fibration transform as in Definition~\ref{def:double_fibration}. Suppose that the Bolker condition holds at $(\hat{z}, \hat{\zeta}, \hat{x}, \hat{\eta}) \in C$. Then, for any $\textbf{f} \in \mathscr{E}'(X,\mathbb{C}^N)$, the following implication holds:
\[
(\hat{z}, \hat{\zeta}) \notin \mathrm{WF}_{\mathrm{a}}(R_W\textbf{f})
\quad \Longrightarrow \quad
(\hat{x}, \hat{\eta}) \notin \mathrm{WF}_{\mathrm{a}}(\textbf{f}),
\]
where $\mathrm{WF}_{\mathrm{a}}$ denotes the analytic wavefront set.
\end{theorem}

Using this theorem, we can prove a local uniqueness result for the real-analytic matrix weighted double fibration transform. In the scalar case, this result appears as \cite[Theorem~1.3]{mazzucchelli2023general}.

\begin{theorem}\label{tm:1.3}
Let $R_W$ be a real-analytic matrix weighted double fibration transform as in Definition~\ref{def:double_fibration}. Suppose that the Bolker condition holds at $(z_0, \zeta_0, x_0, \nu_0) $ $\in C$. 
Suppose that $\Sigma$ is a $C^1$ hypersurface in $X$ such that 
$(x_0, \nu_0) \in N^* \Sigma$. Let $\bsf \in \mathscr{E}'(X, \mathbb{C}^N)$ 
be a distribution that vanishes on one side of $\Sigma$ in a neighborhood of $x_0$. 
If $R_W \bsf(z) = 0$ for $z$ in a neighborhood of $z_0$, then there exists a neighborhood $V_{x_0} \subset X$ of $x_0$ such that $\bsf = 0$ in $V_{x_0}$.
\end{theorem}

\subsection{Applications}
Let $(M,g)$ be a smooth compact manifold of dimension $n \geq 2$ with boundary, and let $\pi \colon \Xi \to M$ be a smooth fiber bundle whose fibers $\Xi_x$ have no boundary. Let $\mathcal{Y}$ be a smooth vector field on $\Xi$ generating the flow $\Phi_t$, and define its horizontal projection $\mathcal{Y}^h := d\pi \circ \mathcal{Y} \subset TM$. For each $z \in \Xi$, define
\begin{equation}
    x_z(t) := \pi\big(\Phi_t(z)\big), \qquad -\tau_-(z) \leq t \leq \tau_+(z),
\end{equation}
where $\tau_-(z)$ and $\tau_+(z)$ denote the backward and forward travel times, respectively.

\begin{definition}\label{def:admi_curve}
Fix an initial-data manifold $\mathcal{G} \subset \Xi$ transversal to $\mathcal Y$. For a point $z \in \mathcal{G}$, the corresponding curve $x_z$ such that $(x_z(0), \dot{x}_z(0)) = z$ is called \emph{admissible} if the following conditions hold:
\begin{enumerate}
    \item[\hypertarget{admi:i}{(i)}] $x_z(t) \in \mathrm{int}(M)$ for all $t \in (-\tau_-(z), \tau_+(z))$;
    \item[\hypertarget{admi:ii}{(ii)}] $t \mapsto x_z(t)$ is injective (no self-intersections);
    \item[\hypertarget{admi:iii}{(iii)}] $\dot{x}_z(t) \neq 0$ for all $t$;
    \item[\hypertarget{admi:iv}{(iv)}] $\tau_\pm(z) < \infty$ (non-trapping);
    \item[\hypertarget{admi:v}{(v)}] if $\dim \mathcal{G} \leq \dim \Xi - 2$, then $V_z(t) + \mathbb{R} \mathcal Y^h = T_{x_z(t)}M$ for all $t$, where
    \[
    V_z(t) := \left\{\, d\pi\left(d\Phi_t(w)\right) \mid w \in T_z \mathcal{G} \,\right\}.
    \]
\end{enumerate}
\end{definition}

Define
\begin{equation}\label{def:Z}
    Z := \left\{\, (z, x_z(t)) \in \mathcal{G} \times M \mid x_z \text{ is an admissible curve} \,\right\}.
\end{equation}
The map $F \colon \mathcal{G} \times \mathbb{R} \to \mathcal{G} \times M$, $F(z,t) = (z, x_z(t))$, is smooth. By \cite[Lemma~4.3(iv)]{mazzucchelli2023general}, conditions \hyperlink{admi:i}{(i)}–\hyperlink{admi:iv}{(iv)} in Definition~\ref{def:admi_curve} imply that $F$ is an injective immersion with closed image in $\mathcal{G} \times M$, so $Z$ is an embedded submanifold. If condition \hyperlink{admi:v}{(v)} also holds, then $Z$ is a double fibration \cite[Lemma~4.3(v)]{mazzucchelli2023general}.

\subsubsection{Matrix weighted geodesic X-ray transforms}

Let $(M,g)$ be a compact Riemannian manifold of dimension $n \geq 2$ with boundary, and assume that $M$ is strictly convex at a point $x_0 \in \partial M$ (see Definition~\ref{def:strict_convex}). Consider $\Xi := SM$ and let $\mathcal Y$ be the geodesic vector field as in Definition~\ref{def:admi_curve}. Let $z_0 = (x_0, v_0) \in \partial_+ SM \cap \partial_- SM = \partial_0 SM$, and let $\mathcal{G}$ be a neighborhood of $z_0$ in $\partial_+ SM$ (see the formula \eqref{eq:influx_out} for the definition of $\partial_\pm SM$). If $\mathcal{G}$ is chosen sufficiently small, then each $x_z$ is a short geodesic, ensuring that conditions \hyperlink{admi:i}{(i)}–\hyperlink{admi:iv}{(iv)} are satisfied. Since $\dim \mathcal{G} = \dim SM - 1$, condition \hyperlink{admi:v}{(v)} holds trivially. Thus, for all $z \in \mathcal{G}$, the curves $x_z$ are admissible.

Let $\omega_{G_z}$ be the orientation form induced on $G_z$ (recall~\eqref{eq:G_z}) and $W \in C^\infty(\mathcal{G} \times M, \mathrm{GL}(N, \mathbb{C}))$. Then
\begin{equation}\label{def_eq:attenuated_double}
    R_W \bsf(z) = \int_{G_z} W(z,x) \bsf(x) \, \mathrm{d}\omega_{G_z} = \int_0^{\tau_+(z)} W(z, x_z(t)) \bsf(x_z(t)) \, \mathrm{d}t,
\end{equation}
where $x_z$ is the unit-speed geodesic through $z \in \mathcal{G}$.

We say that $(M,g)$ is a \emph{real-analytic manifold} if $M$ admits a real-analytic atlas and the metric $g$ is real-analytic. In Proposition~\ref{prop:4.7}, we show that the matrix weighted geodesic ray transform is a matrix weighted double fibration transform. Using this, Theorem~\ref{tm:1.3}, and a vector-valued generalization of the microlocal analytic continuation result from \cite[Theorem~8.5.6']{hormanderiv} (proved as Theorem~\ref{tm:6.1}), we obtain:

\begin{theorem}\label{tm:local:matrix_we}
Let $(M, g)$ be a compact real-analytic Riemannian manifold of dimension $n \geq 2$ with boundary. Suppose that \(\partial M\) is strictly convex at a point \(x_0 \in \partial M\). Let \(z_0 = (x_0, v_0) \in \partial_0 SM\) and \(\mathcal{G}\) be a neighborhood of \(z_0\) in \(\partial_+ SM\).
If $W \in C^\infty(\mathcal{G} \times M, \mathrm{GL}(N, \mathbb{C}))$ is real-analytic and $R_W \bsf(z) = 0$ for all $z \in \mathcal{G}$, then there exists a neighborhood $V_{x_0} \subset M$ of $x_0$ such that $\bsf = 0$ in $V_{x_0}$.
\end{theorem}
This result establishes local injectivity of the matrix-weighted geodesic X-ray transform near a strictly convex boundary point on real-analytic manifolds.

From \eqref{def_eq:attenuated_double}, Theorem~\ref{tm:local:matrix_we}, and the strictly convex foliation constructed in \cite{BGL02}, together with a layer stripping argument, we obtain the following global result:

\begin{corollary}\label{tm:1.4}
Let $(M, g)$ be a compact Riemannian manifold of dimension $n = 2$ with strictly convex boundary and non-trapping geodesic flow. If $M$ and $g$ are analytic, then the matrix weighted geodesic X-ray transform, with a $\mathrm{GL}(N, \mathbb{C})$-valued real-analytic weight, acting on $\mathbb{C}^N$-valued functions on $(M, g)$, is injective.
\end{corollary}
This result was previously established for real-analytic simple manifolds in~\cite[Theorem~1.2]{Zhou:2017}. In that work, the author analyzes the normal operator, which necessitates the global simplicity assumption. This constitutes a key technical distinction between the present work and~\cite{Zhou:2017}.
For compact manifolds of dimension \(\dim M \geq 3\) with strictly convex boundary, assuming the existence of a smooth strictly convex function, a similar injectivity result is proved in~\cite{Paternain_Salo_Uhl_Zhou:2019}.

\subsubsection{Attenuated and nonabelian ray transforms}

A \emph{$\mathbb{C}^{N \times N}$-attenuation} $\mathcal{A}$ on $Z$ is given by an $N \times N$ matrix whose entries are smooth $\mathbb{C}$-valued functions on $Z$. Fix a geodesic curve $x_z \colon [0, \tau_+(z)] \to M$ with $(x_z(0),\dot x_z(0))=z$. To define the \emph{parallel transport} associated with $\mathcal{A}$ along $x_z$, we solve the linear system of ordinary differential equations (ODEs)
\begin{align}\label{eq:parallel_transport_ODE}
\begin{cases}
\frac{d}{dt} s(z, x_z(t)) + \mathcal{A}(z, x_z(t)) s(z, x_z(t)) &= 0, \\
s(z, x_z(\tau_+(z))) &= w \in \mathbb{C}^N.
\end{cases}
\end{align}
The unique solution $s(z, x_z(\cdot))$ defines the linear map
\[
P_{\mathcal{A}}(x_z) \colon \mathbb{C}^N \to \mathbb{C}^N, \quad P_{\mathcal{A}}(x_z)(w) := s(z, x_z(0)).
\]
This map is called the \emph{parallel transport} of $w$ along $x_z$ with respect to $\mathcal{A}$.

Alternatively, one may consider the \emph{fundamental matrix solution} 
\[
U(z, x_z(\cdot)) \colon [0, \tau_+(z)] \to \mathrm{GL}(N, \mathbb{C})
\]
of \eqref{eq:parallel_transport_ODE}, uniquely determined by
\begin{align}\label{eq:fundamental_matrix_solution}
\begin{cases}
\frac{d}{dt} U(z, x_z(t)) + \mathcal{A}(z, x_z(t)) U(z, x_z(t)) &= 0, \\
U(z, x_z(\tau_+(z))) &= \mathrm{Id}.
\end{cases}
\end{align}
Then $P_{\mathcal{A}}(x_z)(w) = U(z, x_z(0))\,w$ for all $w \in \mathbb{C}^N$ and $z \in \partial_+ SM$.

We define the \emph{scattering data} or \emph{nonabelian ray transform of the attenuation} as the map
\[
C_{\mathcal{A}} \colon \partial_+ SM \to \mathrm{GL}(N, \mathbb{C}), \quad C_{\mathcal{A}}(z) := U(z, x_z(0)),
\]
where $U(z, x_z(\cdot))$ is the fundamental matrix solution to \eqref{eq:fundamental_matrix_solution}. In other words, $C_{\mathcal{A}}(z)$ encodes the parallel transport along the geodesic starting from $z$ until it exits $M$.

\begin{definition}[Attenuated ray transform]\label{def:attenuated}
Let $(M, g)$ be a compact Riemannian manifold of dimension $n \geq 2$ with smooth boundary. The attenuated ray transform of $\bsf \in C^{\infty}(Z, \mathbb{C}^N)$ is defined by
\[
I_{\mathcal{A}} \bsf(z) := u^{\bsf}(z, \pi_M(z)), \qquad z \in \partial_+ SM,
\]
where $u^{\bsf} \colon Z \to \mathbb{C}^N$ is the solution to the transport equation
\begin{align*}
\frac{d}{dt} u^{\bsf}(z, x_z(t)) + \mathcal{A}(z, x_z(t)) u^{\bsf}(z, x_z(t)) &= -\bsf(z, x_z(t)), \\
u^{\bsf}(z, x_z(\tau_+(z))) &= 0.
\end{align*}
\end{definition}

We now state a local uniqueness result for the real-analytic attenuated ray transform.

\begin{theorem}\label{tm:local_attenuated}
Let $(M, g)$ be a compact real-analytic Riemannian manifold of dimension $n \geq 2$ with boundary. Suppose that \(\partial M\) is strictly convex at a point \(x_0 \in \partial M\). Let \(z_0 = (x_0, v_0) \in \partial_0 SM\) and \(\mathcal{G}\) be a neighborhood of \(z_0\) in \(\partial_+ SM\).
Suppose $\mathcal{A} \in C^\infty(\mathcal{G} \times M, \mathbb{C}^{N \times N})$ is a real-analytic matrix-valued function. If the attenuated ray transform $I_{\mathcal{A}} \bsf(z) = 0$ for all $z \in \mathcal{G}$, then there exists a neighborhood $V_{x_0} \subset M$ of $x_0$ such that $\bsf = 0$ in $V_{x_0}$.
\end{theorem}

A smooth matrix-valued function $\Phi\in C^{\infty}(M, \mathbb{C}^{N \times N})$ is called a \emph{Higgs field} on $M$. When $\Phi$ is viewed along geodesics, it defines an attenuation on $Z$ via $\Phi(z,x_z(t)):=\Phi(x_z(t))$.

 Finally, we state a \emph{global} uniqueness theorem for the real-analytic nonabelian ray transform, assuming strict convexity at a single boundary point and local data.
\begin{corollary}\label{tm:non_abelian}
Let $(M, g)$ be a compact real-analytic Riemannian manifold of dimension $n \geq 2$ with boundary. Suppose that \(\partial M\) is strictly convex at a point \(x_0 \in \partial M\). Let \(z_0 = (x_0, v_0) \in \partial_0 SM\) and \(\mathcal{G}\) be a neighborhood of \(z_0\) in \(\partial_+ SM\). Let $\Phi, \Psi \in C^{\infty}(M, \mathbb{C}^{N \times N})$ be real-analytic Higgs fields. If
\[
C_{\Phi}(z) = C_{\Psi}(z) \quad \text{for all } z\in \mathcal G,
\]
then $\Phi = \Psi$ in $M$.
\end{corollary}
\section{Preliminaries}\label{sec:pre}

We begin by recalling some notions and definitions, primarily following~\cite{hormanderi,hormanderiv,Duistermaat:FIO:book}, that will be used throughout the article.

Let \(X\) and \(Y\) be smooth manifolds of dimensions \(n_X\) and \(n_Y\), respectively, and let \(N \geq 1\). Throughout, we work with the trivial complex vector bundle \(\mathbb{C}^N\) over both \(X\) and \(Y\), rather than general vector bundles. Let \(\Omega_X^{1/2}\) and \(\Omega_Y^{1/2}\) denote the bundles of half-densities on \(X\) and \(Y\), respectively. We define the following function spaces:
\begin{align*}
    C^{\infty}_c(Y; \Omega_Y^{1/2} \otimes \mathbb{C}^N) 
    &:= \left\{\,\text{smooth, compactly supported sections of } \Omega_Y^{1/2} \otimes \mathbb{C}^N \,\right\}, \\
    % \mathscr{D}'(Y; \Omega_Y^{1/2} \otimes \mathbb{C}^N) 
    % &:= \mathcal{L}\left(C^{\infty}_c(Y; \Omega_Y^{1/2} \otimes \mathbb{C}^N), \mathbb{C}\right), \\
     \mathscr{D}'(Y; \Omega_Y^{1/2} \otimes \mathbb{C}^N) 
    &:= \left\{\, \text{continuous linear functionals on } C^{\infty}_c(Y; \Omega_Y^{1/2} \otimes \mathbb{C}^N)\,\right\}, \\
    \mathscr{E}'(Y; \Omega_Y^{1/2} \otimes \mathbb{C}^N) 
    &:= \left\{\, f \in \mathscr{D}'(Y; \Omega_Y^{1/2} \otimes \mathbb{C}^N) \mid \operatorname{supp}(f) \Subset Y \,\right\}.
\end{align*}
Fix coordinate patches \(U \subset Y\) with coordinates \(y = (y_1, \ldots, y_{n_Y})\) and \(V \subset X\) with coordinates \(x = (x_1, \ldots, x_{n_X})\), and choose the standard trivializing half-densities
\[
|dx|^{1/2} := \left|dx_1 \wedge \cdots \wedge dx_{n_X}\right|^{1/2}, \qquad 
|dy|^{1/2} := \left|dy_1 \wedge \cdots \wedge dy_{n_Y}\right|^{1/2}.
\]

For example, if \(\phi \in C^{\infty}_c(Y; \Omega_Y^{1/2} \otimes \mathbb{C}^N)\) has support contained in \(U\), then it can be uniquely written in the form
\[
\phi(y) = (\phi_1(y), \dots, \phi_N(y))^{\mathrm{T}} |dy|^{1/2},
\]
where each \(\phi_i \in C^{\infty}_c(Y)\) for \(1 \leq i \leq N\).

\begin{remark}
Let \(\mu\) be a fixed nowhere-vanishing smooth density on \(Y\). Then the map
\[
C^{\infty}_c(Y; \Omega_Y^{1/2} \otimes \mathbb{C}^N) \to C^{\infty}_c(Y; \mathbb{C}^N), \quad u \mapsto \mu^{-1/2} u
\]
is an isomorphism, with inverse given by \(f \mapsto f \mu^{1/2}\). This identification allows us to work with vector-valued functions instead of half-densities when convenient.
\end{remark}

For \(u \in \mathscr{D}'(Y; \Omega_Y^{1/2} \otimes \mathbb{C}^N)\), the duality pairing with \(\phi \in C^{\infty}_c(Y; \Omega_Y^{1/2} \otimes \mathbb{C}^N)\) can be written as
\begin{equation}\label{eq:duality}
    \langle u, \phi \rangle_Y = \sum_{j=1}^N u_j(\phi_j)
\end{equation}
where
\[
u_j(g) := u(g e_j), \quad \forall g \in C^{\infty}_c(Y; \Omega_Y^{1/2} \otimes \mathbb{C}),
\]
and \(\{e_j\}_{j=1}^N\) denotes the standard basis of \(\mathbb{C}^N\) over \(\mathbb{C}\).
If each \(u_j\) is represented by a locally integrable function \(f_j \in L^1_{\mathrm{loc}}(Y)\), then
\[
u_j(\phi_j) = \int_U f_j(y) \phi_j(y) \, |dy|.
\]
A submanifold \(L \subset T^*Y \setminus \{0\}\) is called \emph{conic} if \((y, \xi) \in L\) implies \((y, t\xi) \in L\) for all \(t > 0\).
\subsection{Analytic microlocal analysis}
We now recall the definition of the analytic wavefront set, denoted \(\WF\), following~\cite[Chapter~VIII]{hormanderi} and~\cite{Sjostrand:1982}. For a distribution
\[
u = (u_1, \dots, u_N)^{\mathrm{T}} \in \mathscr{D}'(Y; \Omega_Y^{1/2} \otimes \mathbb{C}^N),
\]
the analytic wavefront set is defined componentwise as
\begin{equation}\label{eq:wave_front_vector}
    \WF(u) := \bigcup_{j=1}^N \WF(u_j) \subset T^*Y \setminus \{0\}.
\end{equation}
Each \(\WF(u_j)\) is a closed conic subset of \(T^*Y \setminus \{0\}\). Given a real-analytic coordinate chart \(\kappa: Y_\kappa \to \mathbb{R}^n\), the wavefront set transforms under pullback as
\[
\WF(u_j) \cap T^*(Y_\kappa) = \kappa^* \left( \WF(u_j \circ \kappa^{-1}) \right),
\]
where the right-hand side denotes the Euclidean analytic wavefront set. The invariance of \(\WF(u_j)\) under real-analytic diffeomorphisms (see~\cite[Theorem~8.5.1]{hormanderi}) ensures that this definition is coordinate-independent.

Let \(\tilde{X} \subset \mathbb{R}^n\) be open. For any \(v \in \mathscr{D}'(\tilde{X})\), the \emph{analytic singular support} \(\operatorname{sing\,supp}_A v\) is the smallest closed subset of \(\tilde{X}\) outside of which \(v\) is real-analytic.

We define the analytic wavefront set via the Fourier--Bros--Iagolnitzer (FBI) transform, following \cite{hormanderi,Sjostrand:1982}. Let $U\subseteq \R^m$ be a bounded open set, and identify
\[
T^*U \cong U \times \mathbb{R}^m, \qquad u = (z, \zeta), \quad z \in U,\; \zeta \in \mathbb{R}^m.
\]
Define the Gaussian
\[
c_m = 2^{-m/2} \pi^{-3m/4}, \qquad M(z) = c_m e^{-\frac{1}{2} |z|^2}, \quad z \in \mathbb{R}^m.
\]
For \(\lambda > 0\), define the Gaussian wave packet centered at \(u = (z, \zeta) \in T^*U\) by
\[
M_u^\lambda(w) := \lambda^{3m/4} e^{i\lambda\, w \cdot \zeta} \, M\big( \lambda^{1/2}(w - z) \big), \qquad w \in \mathbb{R}^m.
\]
Let \(f \in \mathscr{E}'(\widehat{U})\), where \(\widehat{U} \Subset U\). The FBI transform of \(f \in \mathscr{D}'(\mathbb{R}^m)\) with respect to the Gaussian \(M\) is defined by
\[
(L_M^\lambda f)(u) := \int_{\mathbb{R}^m} f(w)\, \overline{M_u^\lambda(w)}\, dw.
\]

According to~\cite[Theorem~9.6.3]{hormanderi}, for \((z_0, \zeta_0) \in T^*\mathbb{R}^m \setminus \{0\}\), we have
\begin{equation}\label{eq:FBI_wavefront}
    (z_0, \zeta_0) \notin \WF(f) \quad \text{if and only if} \quad \exists\, \varepsilon > 0 \text{ such that } (L_M^\lambda f)(v) = \mathcal{O}(e^{-\varepsilon \lambda})
\end{equation}
uniformly for \(v\) in a conic neighborhood of \((z_0, \zeta_0)\) as \(\lambda \to \infty\). The analyatic wavefront set $\operatorname{WF}_a(f)$ is a closed conic subset of $T^*\R^n\setminus\{0\}$, and its projection onto the base coincides with $\operatorname{sing supp}_Af$.  An equivalent characterization can be found in \cite[Section~8.4]{hormanderi}, \cite[Appendix~A.8.2]{Stefanov:Uhlmann:MicrolocalBook}.

The following theorem provides an equivalent characterization of the analytic wavefront set using general analytic phase functions, generalizing the FBI transform defination above.
\begin{theorem}[{\cite[Section~6]{Sjostrand:1982}}]\label{thm:FBI}
Let \((x_0, \xi_0) \in T^*\mathbb{R}^n \setminus \{0\}\) be fixed. Let \(\phi(x, \xi, y)\) be a \emph{phase function} defined in a neighbourhood $U \subset \mathbb{C}^{3n}$ of $(x_0, \xi_0, x_0)$, holomorphic in all variables and satisfying 
\[\varphi(x,\xi,x)=0, \qquad \partial_y\varphi(x,\xi,x)=-\xi,\qquad \operatorname{Im}\varphi(x,\xi,y)\geq \frac{|x-y|^2}{C}\]
for some $C>0$, whenever $(x,\xi,y)$ real in $U$. Such a $\varphi$ is called an \emph{analytic phase function}. 
Let \(a(x, \xi, y; \lambda)\) be a classical analytic symbol that is elliptic in \(U\) (see~\cite[Theorem~1.5]{Sjostrand:1982}). Then, for any $u\in \mathcal D'(\R^n)$, the following are equivalent:
\begin{enumerate}
  \item \((x_0, \xi_0) \notin \WF(u)\).
  \item There exists \(\chi \in C_0^\infty(\mathbb{R}^n)\) with \(\chi(x_0) = 1\) and a constant \(C > 0\) such that
  \[
  \int_{\mathbb{R}^n}
    e^{i\lambda \phi(x, \xi, y)}\,
    a(x, \xi, y; \lambda)\,
    \chi(y)\, u(y)\, dy
    = \mathcal{O}(e^{-\lambda/C}) \quad \text{as } \lambda \to \infty,
  \]
  uniformly for \((x, \xi)\) in a real conic neighborhood of \((x_0, \xi_0)\).
\end{enumerate}
\end{theorem}

\subsection{Matrix-valued Fourier integral operators}

Let \(\varphi \in C^{\infty}_c(X; \Omega_X^{1/2} \otimes \mathbb{C}^N)\) and \(u \in C^{\infty}_c(Y; \Omega_Y^{1/2} \otimes \mathbb{C}^N)\). The \emph{external tensor product} is defined by
\begin{equation}
    (\varphi \boxtimes u)(x, y) := \varphi(x) \otimes u(y) = \left( \varphi_i(x) u_j(y) \right)_{1 \le i,j \le N} |dx|^{1/2} |dy|^{1/2}.
\end{equation}

Let
\[
K \in \mathscr{D}'(X \times Y; \Omega_{X \times Y}^{1/2} \otimes \mathrm{Hom}(\mathbb{C}^N, \mathbb{C}^N)), \quad K = (K_{ij})_{1 \le i,j \le N}
\]
be a distributional kernel valued in \(N \times N\) complex matrices with half-density coefficients. It defines a continuous linear operator via the bilinear pairing:
\begin{align}\label{eq:K:operator}
\begin{split}
    \mathcal{K} : C^{\infty}_c(Y; \Omega_Y^{1/2} \otimes \mathbb{C}^N) &\to \mathscr{D}'(X; \Omega_X^{1/2} \otimes \mathbb{C}^N), \\
    \langle \mathcal{K}u, \varphi \rangle_X &:= \langle K, \varphi \boxtimes u \rangle_{X \times Y},
\end{split}
\end{align}
where \(\varphi \in C^{\infty}_c(X; \Omega_X^{1/2} \otimes \mathbb{C}^N)\), and the pairing \(\langle \cdot, \cdot \rangle_X\) denotes the canonical duality between distributions and test functions on \(X\) (see~\eqref{eq:duality}).

Let \((x_1, \dots, x_{n_X})\) be local coordinates on \(X\), and \((x_1, \dots, x_{n_X}, \xi_1, \dots, \xi_{n_X})\) the induced coordinates on \(T^*X\). The \emph{canonical 1-form} on \(T^*X\) is given by \(\alpha = \sum_{j=1}^{n_X} \xi_j\, dx_j\). By~\cite[Proposition~3.7.1]{Duistermaat:FIO:book}, a closed \(n_X\)-dimensional submanifold \(L \subset T^*X \setminus \{0\}\) is a \emph{conic Lagrangian manifold} if and only if \(\alpha|_L = 0\).

Let \(\Lambda \subset T^*(X \times Y) \setminus \{0\}\) be a closed conic Lagrangian submanifold. For \(m \in \mathbb{R}\), define
\[
I^m(X \times Y, \Lambda; \Omega_{X \times Y}^{1/2} \otimes \mathrm{Hom}(\mathbb{C}^N, \mathbb{C}^N))
\]
to be the space of matrix-valued Lagrangian distributions of order \(m\) associated with \(\Lambda\) (see~\cite[Definition~25.1.1, Section~25.2]{hormanderiv}).

\begin{remark}\label{rm:identification}
There is a canonical identification
\[
T^*(X \times Y) \setminus \{0\} \cong (T^*X \setminus \{0\}) \times (T^*Y \setminus \{0\}) \cup (T^*X \setminus \{0\}) \times \{0\} \cup \{0\} \times (T^*Y \setminus \{0\}),
\]
given by the splitting \((x, y; \xi, \eta) \mapsto (x; \xi, y; \eta)\).
\end{remark}

The \emph{twisted canonical relation} associated with \(\Lambda\) is
\[
\Lambda' := \left\{\, (x, \xi, y, -\eta) \mid (x, \xi, y, \eta) \in \Lambda \,\right\},
\]
which defines a canonical relation from \(T^*Y \setminus \{0\}\) to \(T^*X \setminus \{0\}\).

\begin{definition}[Matrix-valued Fourier integral operator, {\cite[Definition~25.2.1]{hormanderiv}}]\label{def:FIO-matrix}
A continuous operator
\[
\mathcal{K} : C^{\infty}_c(Y; \Omega_Y^{1/2} \otimes \mathbb{C}^N) \to \mathscr{D}'(X; \Omega_X^{1/2} \otimes \mathbb{C}^N)
\]
is called a \emph{Fourier integral operator of order \(m\) associated with \(\Lambda'\)} if its kernel \(K\) lies in
\[
I^m(X \times Y, \Lambda'; \Omega_{X \times Y}^{1/2} \otimes \mathrm{Hom}(\mathbb{C}^N, \mathbb{C}^N)).
\]
We write
\[
\mathcal{K} \in I^m(X, Y, \Lambda'; \Omega_{X \times Y}^{1/2} \otimes \mathrm{Hom}(\mathbb{C}^N, \mathbb{C}^N)).
\]
\end{definition}
\begin{lemma}[Vector-valued version of {\cite[Theorem~8.5.5]{hormanderi}}]\label{thm:8.5.5-vector}
Let \(u \in \mathscr{E}'(Y; \Omega_Y^{1/2} \otimes \mathbb{C}^N)\) and let
\[
K \in \mathscr{D}'(X \times Y; \Lambda'; \Omega_{X \times Y}^{1/2} \otimes \mathrm{Hom}(\mathbb{C}^N, \mathbb{C}^N))
\]
be a matrix-valued distributional kernel associated with a conic Lagrangian submanifold \(\Lambda' \subset T^*(X \times Y) \setminus \{0\}\). Suppose that
\[
\WF(u) \cap \WF'(K)_Y = \emptyset,
\]
where
\begin{align*}
    \WF'(K) &:= \left\{\, (x, \xi, y, -\eta) \mid (x, y, \xi, \eta) \in \WF(K) \,\right\}, \\
    \WF'(K)_Y &:= \left\{\, (y, \eta) \mid \exists\, x \in X \text{ such that } (x, y, 0, -\eta) \in \WF(K) \,\right\}.
\end{align*}
Then the operator \(\mathcal{K}\) defined in~\eqref{eq:K:operator} can be extended uniquely to $\mathcal E'(Y; \Omega_Y^{1/2} \otimes \mathbb{C}^N)$, and the extended operator satisfies
\[
\WF(\mathcal{K}u) \subset \WF(K)_X \cup \left( \WF'(K) \circ \WF(u) \right),
\]
where
\[
\WF(K)_X := \left\{\, (x, \xi) \mid \exists\, y \in Y \text{ such that } (x, y, \xi, 0) \in \WF(K) \,\right\}.
\]
\end{lemma}

\begin{proof}
By the definition of the wavefront set for vector-valued distributions (see~\eqref{eq:wave_front_vector}), we have
\begin{align*}
    \WF(\mathcal{K}u)
    &= \bigcup_{i=1}^N \WF((\mathcal{K}u)_i)
    = \bigcup_{i=1}^N \WF\left( \sum_{j=1}^N \langle K_{ij}, u_j \rangle_Y \right)
    \subset \bigcup_{i,j=1}^N \WF\left( \langle K_{ij}, u_j \rangle_Y \right).
\end{align*}
For each fixed pair \(i, j \in \{1, \dots, N\}\), we apply the scalar-valued result~\cite[Theorem~8.5.5]{hormanderi} to obtain
\[
\WF\left( \langle K_{ij}, u_j \rangle_Y \right)
\subset \WF(K_{ij})_X \cup \left( \WF'(K_{ij}) \circ \WF(u_j) \right),
\]
where
\[
\WF'(K_{ij}) := \left\{\, (x, \xi; y, \eta) \mid (x, y, \xi, -\eta) \in \WF(K_{ij}) \,\right\} \subset T^*X \times T^*Y,
\]
and
\[
\WF'(K_{ij}) \circ \WF(u_j) := \left\{\, (x, \xi) \mid \exists\, (y, \eta) \in \WF(u_j) \text{ such that } (x, \xi; y, \eta) \in \WF'(K_{ij}) \,\right\}.
\]
Taking the union over all \(i, j\), we conclude
\[
\WF(\mathcal{K}u) \subset \bigcup_{i,j=1}^N \left( \WF(K_{ij})_X \cup \left( \WF'(K_{ij}) \circ \WF(u_j) \right) \right),
\]
which is contained in
\[
\WF(K)_X \cup \left( \WF'(K) \circ \WF(u) \right),
\]
as claimed.
\end{proof}

\begin{remark}
In addition to the assumptions of Lemma~\ref{thm:8.5.5-vector}, suppose that the Lagrangian \(\Lambda\) satisfies
\[
\Lambda \subset \left(T^*X \setminus \{0\}\right) \times \left(T^*Y \setminus \{0\}\right).
\]
Then it follows that
\[
\WF(K)_X = \WF'(K)_Y = \emptyset,
\]
and consequently, the analytic wavefront set of \(\mathcal{K}u\) satisfies the sharper inclusion
\[
\WF(\mathcal{K}u) \subset \WF'(K) \circ \WF(u).
\]
\end{remark}
\subsection{Geometric preliminaries}

In this subsection, we recall some geometric notions relevant to our analysis. For further details, we refer to~\cite[Chapter~3]{GIP2D}. Let \((M, g)\) be a compact, connected, oriented Riemannian manifold of dimension \(n \geq 2\) with smooth boundary.

The \emph{unit tangent bundle} is defined by
\[
S M := \left\{\, (x, v) \in T M \mid \langle v, v \rangle_{g(x)} = 1 \,\right\}.
\]
The boundary of \(S M\) is given by
\[
\partial S M := \left\{\, (x, v) \in S M \mid x \in \partial M \,\right\}.
\]
We define the \emph{influx} and \emph{outflux} boundaries of \(S M\) as
\begin{equation}\label{eq:influx_out}
    \partial_{\pm} S M := \left\{\, (x, v) \in \partial S M \mid \pm \langle v, \nu(x) \rangle_{g(x)} \geq 0\, \right\},
\end{equation}
where \(\nu(x)\) denotes the inward-pointing unit normal to \(\partial M\) at \(x\). The \emph{glancing region} is defined as
\[
\partial_0 S M := \partial_{+} S M \cap \partial_{-} S M = S(\partial M).
\]

\begin{definition}
The manifold \((M, g)\) is said to be \emph{non-trapping} if for every \(z \in S M\), the geodesic \(x_z\) issued from \(z\) satisfies \(\tau_+(z) < \infty\), where \(\tau_+\) denotes the forward exit time.
\end{definition}

\begin{definition}\label{def:strict_convex}
The boundary \(\partial M\) of a Riemannian manifold \((M, g)\) is said to be \emph{strictly convex at a point} \(x \in \partial M\) if the second fundamental form \(\Pi_x\) is positive definite. The boundary is said to be \emph{strictly convex} if this condition holds at every point \(x \in \partial M\). The second fundamental form is the bilinear form on \(T_x \partial M\) defined by
\[
\Pi_x(v, w) := -\left( \nabla_v \nu, w \right)_g, \qquad v, w \in T_x \partial M,
\]
where \(\nu\) is the inward unit normal and \(\nabla\) is the Levi-Civita connection.
\end{definition}

According to~\cite[Lemma~3.1.12]{GIP2D}, let \((M, g)\) be a compact Riemannian manifold with strictly convex boundary \(\partial M\), and let \((M_1, g)\) be a closed extension of \(M\). Then for every \(z \in \partial_{+} S M \setminus \partial_0 S M\), the corresponding geodesic \(x_z(t)\) remains entirely in the interior \(\operatorname{int}(M)\) for all \(t \in (-\tau_-(z), \tau_+(z))\), satisfying condition~\hyperlink{admi:i}{(i)} of Definition~\ref{def:admi_curve}.

Moreover, any geodesic that is tangent to \(\partial M\) at some point remains outside \(M\) for sufficiently small positive and negative times.
\section{Proof of the Main Result}

We recall the setting introduced in Section~\ref{sec:Intro}. Let $\mathcal{G}$ and $X$ be oriented manifolds without boundary, with $\dim(\mathcal{G}) = m$ and $\dim(X) = n$, and let $Z \subset \mathcal{G} \times X$ be an oriented embedded submanifold. Suppose that the natural projections
\[
\pi_\mathcal{G} \colon Z \to \mathcal{G}, \qquad \pi_X \colon Z \to X
\]
are submersions, and that
\[
m + n > \dim(Z) > m \geq n.
\]
We write $\dim(Z) = m + n'$ for some $n' \in \{1, \dotsc, n-1\}$, and define $n'' := n - n'$. Let $(N^* Z \setminus 0)'$ denote the twisted conormal bundle of $Z$ (see~\eqref{eq:twisted_conormal} for the definition), and define the natural projections
\[
\pi_L \colon (N^* Z \setminus 0)' \to T^* \mathcal{G}, \qquad
\pi_R \colon (N^* Z \setminus 0)' \to T^* X.
\]
Suppose the Bolker condition holds at a point $(z_0, \zeta_0, x_0, \eta_0) \in (N^* Z \setminus 0)'$, that is,
\[
\pi_L^{-1}(z_0, \zeta_0) = \{(z_0, \zeta_0, x_0, \eta_0)\}, \quad \text{and} \quad d\pi_L|_{(z_0, \zeta_0, x_0, \eta_0)} \text{ is injective}.
\]

We now work in local coordinates near $(z_0, \zeta_0, x_0, \eta_0)$, identifying neighborhoods with open subsets of $\mathbb{R}^m \times \mathbb{R}^m \times \mathbb{R}^n \times \mathbb{R}^n$. Let
\[
w = (w'', w') \in \mathbb{R}^{n''} \times \mathbb{R}^{m - n''}, \qquad y \in \mathbb{R}^n.
\]
By~\cite[Lemma 5.7]{mazzucchelli2023general}, there exists a coordinate representation for the analytic double fibration \(Z\), which is an analytic submanifold of \(V \times U\) of dimension \(m + n'\). More precisely, there exist an open neighborhood 
\(\Omega \subset V \times U\) of \((z_0, x_0)\), an open set
\(\widetilde{\Omega} \subset \mathbb{R}^{m} \times \mathbb{R}^{n'+n''}\), and a real-analytic diffeomorphism
\[
(w'', w', y) \mapsto \big(b(w'', w', y),\, y\big) \in V \times U, \qquad 
\widetilde{\Omega} \to \Omega,
\]
such that
\begin{equation}\label{eq:Z-analytic}
    Z \cap \Omega 
=
\left\{\,\big(b(0, w', y), y\big) \;\middle|\;
w' \in \mathcal{U},\;
y \in U\, \right\},
\end{equation}
where \(\mathcal{U} \subset \mathbb{R}^{m - n''}\) is a neighborhood of the origin. Here, $b:\tilde \Omega\to V$ is a real-analytic map.

For an open set \(\Omega \subset \mathbb{R}^k\), we denote by \(\Omega_{\mathbb{C}} \subset \mathbb{C}^k\) an open neighborhood such that \(\Omega_{\mathbb{C}} \cap \mathbb{R}^k = \Omega\). The coordinate system~\eqref{eq:Z-analytic} (possibly after shrinking \(\Omega\) and \(\widetilde{\Omega}\)) extends holomorphically to a coordinate system \(\widetilde{\Omega}_{\mathbb{C}} \to \Omega_{\mathbb{C}}\), such that \(Z_{\mathbb{C}} \cap \Omega_{\mathbb{C}}\) is again given by the condition \(w'' = 0\), as in~\eqref{eq:Z-analytic}.

From now on, we denote \(b(w', y) := b(0, w', y)\) for notational simplicity. In these coordinates, we define the associated \emph{phase function} \[
\Psi \colon \mathcal{U} \times U \times \widetilde{V} \to \mathbb{C}, \quad \text{where } \mathcal{U} \subset \mathbb{R}^{m - n''},\; U \subset \mathbb{R}^n,\; \widetilde{V} \subset \mathbb{R}^{2m},
\]
by
\begin{equation}\label{def:Psi}
    \Psi(w', y, z, \zeta) := - b(w', y) \cdot \zeta + \frac{i}{2} \lvert b(w', y) - z \rvert^2.
\end{equation}

From~\cite[Lemma 5.5]{mazzucchelli2023general}, under the assumption that the Bolker condition holds at \((z_0, \zeta_0, x_0, \eta_0) \in (N^* Z \setminus 0)'\), there exists a neighborhood \(\tilde{C}\) of \((z_0, \zeta_0, x_0, \eta_0)\) in \((N^* Z \setminus 0)'\) such that the map
\begin{equation}\label{eq:lm:5.5}
    \mu := \pi_R \circ \pi_L^{-1} \colon \pi_L(\tilde{C}) \to \pi_R(\tilde{C})
\end{equation}
is an analytic surjective submersion.

By the implicit function theorem in the analytic category, there exists an analytic map
\begin{equation}\label{eq:chi+}
    \mu^{+} \colon \pi_R(\tilde{C}) \to \pi_L(\tilde{C}) \quad \text{such that} \quad \mu \circ \mu^{+} = \mathrm{Id}.
\end{equation}
\begin{theorem}\label{tm:5.2}
Let \(Z \subset \mathcal{G} \times X\) be a double fibration, and let \(R_W\) be a real-analytic matrix-weighted double fibration ray transform as in Definition~\ref{def:double_fibration}. Assume that the Bolker condition holds at \((z_0, \zeta_0, x_0, \eta_0) \in (N^* Z \setminus 0)'\). Then there exists a neighborhood \(\hat{U} \Subset U\) containing \(x_0\) such that for all \(\bsf \in \mathscr{E}'(\hat{U}; \mathbb{C}^N)\),
\[
(z_0, \zeta_0) \notin \mathrm{WF}_{\mathrm{a}}(R_W \bsf) \quad \Longrightarrow \quad (x_0, \eta_0) \notin \mathrm{WF}_{\mathrm{a}}(\bsf),
\]
where \(\mathrm{WF}_{\mathrm{a}}\) denotes the analytic wavefront set.
\end{theorem}

To prove this theorem, we rely on the following key lemma, which enables the application of Sjöstrand’s analytic microlocal analysis.

\begin{lemma}[{\cite[Proposition~5.6]{mazzucchelli2023general}}]\label{theorem:MST}
There exists a small complex neighborhood 
\[
\tilde{U} \times \tilde{\mathcal{V}} \subset \mathbb{C}^n \times \mathbb{C}^{2m}
\]
of \((x_0, z_0, \zeta_0)\) such that for any fixed \((y, z, \zeta) \in \tilde{U} \times \tilde{\mathcal{V}}\), the map
\[
w' \mapsto \Psi(w', y, z, \zeta),
\]
defined near a neighborhood \(\mathcal{U} \subset \mathbb{R}^{m - n''}\) of the origin (see~\eqref{eq:Z-analytic}, \eqref{def:Psi}), has a unique nondegenerate critical point \(w'_c(y, z, \zeta) \in \mathcal{U}_{\mathbb{C}}\) depending holomorphically on \((y, z, \zeta)\). Define
\[
\psi(y, z, \zeta) := \Psi(w'_c(y, z, \zeta), y, z, \zeta).
\]
Then \(\psi\) satisfies the following properties:
\begin{enumerate}
    \item[(i)] \(\psi\big(\pi(\mu(z, \zeta)), z, \zeta\big) + z \cdot \zeta = 0\), where \(\mu := \pi_R \circ \pi_L^{-1}\) (see~\eqref{eq:lm:5.5});
    \item[(ii)] \(\left(y, -d_y \psi(y, z, \zeta)\right)\big|_{y = \pi(\mu(z, \zeta))} = \mu(z, \zeta)\), in particular,
    \[
    \left(y, -d_y \psi(y, z, \zeta)\right)\big|_{y = \pi(\mu(z_0, \zeta_0))} = (x_0, \eta_0);
    \]
    \item[(iii)] \(\operatorname{Im} \psi(y, z, \zeta) \geq c \lvert y - \pi(\mu(z, \zeta)) \rvert^2\) for real-valued \((y, z, \zeta)\) and some constant \(c > 0\).
\end{enumerate}
\end{lemma}\begin{proof}[Proof of Theorem~\ref{tm:5.2}]
Assume \((z_0, \zeta_0) \notin \mathrm{WF}_{\mathrm{a}}(R_W \bsf)\). By~\eqref{eq:FBI_wavefront}, there exists \(\varepsilon > 0\) such that the FBI transform of \(R_W \bsf\) satisfies
\[
L_M^\lambda R_W \bsf(z, \zeta)
:=
C_{m} \lambda^{3m/4}
\int_{\mathbb{R}^m}
e^{-i\lambda w \cdot \zeta - \lambda \lvert w - z \rvert^2 / 2}
R_W \bsf(w) \, d\omega_\mathcal{G}(w)
=
\mathcal{O}(e^{-\varepsilon \lambda}) \, \mathcal{I}
\]
as \(\lambda \to \infty\), uniformly for \((z, \zeta)\) near \((z_0, \zeta_0)\), where \(\mathcal{I}\) denotes the \(\mathbb{C}^N\)-valued vector with each entry equal to 1. We consider \((z, \zeta)\) in a sufficiently small neighborhood \(\mathcal{V}\) of \((z_0, \zeta_0)\).

Since
\[
\left| e^{-i\lambda w \cdot \zeta - \lambda \lvert w - z \rvert^2 / 2} \right| \leq e^{-2\varepsilon \lambda}
\]
for \(z \in \pi(\mathcal{V})\) and \(w\) satisfying \(\lvert z - w \rvert \geq 2\sqrt{\varepsilon}\), we may assume that \(\operatorname{supp}(R_W \bsf)\) is contained in a small neighborhood of \(z_0\).

From~\eqref{eq:Z-analytic}, we have that $Z$ may be presented locally by the coordinates $\{(b(w', y), y)\}$. Using this and Fubini's theorem with the formula
\[
d\omega_{G_w}(y) \, d\omega_{\mathcal{G}}(w)
=
d\omega_Z(w, y)
=
d\omega_{H_y}(w) \, d\omega_X(y), \quad H_y := \pi_\mathcal{G} \circ \pi_X^{-1}(y),
\]
we deduce that
\begin{equation}\label{equation:fbi}
L_M^\lambda R_W \bsf(z, \zeta)
=
\int_{\mathbb{R}^n} K_\lambda(y, z, \zeta) \bsf(y) \, dy
=
\mathcal{O}(e^{-\varepsilon \lambda}) \, \mathcal{I},
\end{equation}
where (see~\cite[Proposition 5.4]{mazzucchelli2023general})
\begin{equation}\label{equation:kernel1}
K_\lambda(y, z, \zeta)
=
C_{n, m} \lambda^{3m/4}
\int_{ \mathcal U}
e^{i \Psi(w', y, z, \zeta)}
W_1(w', y) \, dw',
\end{equation}
with
\[
W_1(w', y) := W\big(b(w', y), y\big),
\]
and \(\Psi\) defined in~\eqref{def:Psi}.

Applying the analytic stationary phase method (steepest descent) from~\cite[Théorème~2.8]{Sjostrand:1982} (see also~\cite[Theorem~2.3.4]{Hitrik:Sjostrand:2018}) to~\eqref{equation:kernel1}, we obtain
\begin{equation}\label{equation:kernel2}
K_\lambda(y, z, \zeta)
=
C_{n, N} \lambda^{5m/4 - n''/2} e^{i \psi(y, z, \zeta)}
\left\{
W_2(y, z, \zeta) + \mathcal{O}(\lambda^{-1}) \operatorname{Id}
\right\}
+ \mathcal{O}(e^{-\varepsilon \lambda}) \operatorname{Id},
\end{equation}
uniformly for complex \((y, z, \zeta)\) near \((x_0, z_0, \zeta_0)\), where
\[
W_2(y, z, \zeta)
=
W_1\big(w'_c(y, z, \zeta), y\big)
=
W\left(
\left( b(w'_c(y, z, \zeta), y)\right), y
\right),
\]
and \(\psi\) is as defined in Lemma~\ref{theorem:MST}.

Substituting~\eqref{equation:kernel2} into~\eqref{equation:fbi}, we obtain
\[
\int_{\mathbb{R}^n}
e^{i \lambda \psi(y, z, \zeta)}
\left\{
W_2(y, z, \zeta) + \mathcal{O}(\lambda^{-1}) \operatorname{Id}
\right\}
\bsf(y) \, dy
=
\mathcal{O}(e^{-\varepsilon \lambda}) \, \mathcal{I}
\]
uniformly for complex \(( z, \zeta)\) near \((z_0, \zeta_0)\).

Multiplying by \(e^{i \lambda z \cdot \zeta}\), we obtain
\begin{equation}\label{equation:bargmann}
\int_{\mathbb{R}^n}
e^{i \lambda (\psi(y, z, \zeta) + z \cdot \zeta)}
\left\{
W_2(y, z, \zeta) + \mathcal{O}(\lambda^{-1}) \operatorname{Id}
\right\}
\bsf(y) \, dy
=
\mathcal{O}(e^{-\varepsilon \lambda}) \, \mathcal{I}
\end{equation}
uniformly for complex \(( z, \zeta)\) near \((z_0, \zeta_0)\).

Since \(\mu \circ \mu^+ = \operatorname{Id}\) (cf.~\eqref{eq:lm:5.5} and~\eqref{eq:chi+}), define the modified phase function
\[
\tilde{\psi}(y, x, \eta) := \left( \psi(y, z, \zeta) + z \cdot \zeta \right)\big|_{(z, \zeta) = \mu^+(x, \eta)}
\]
for complex \((x, \eta)\) near \((x_0, \eta_0)\).
Then Lemma~\ref{theorem:MST} implies that
\[
d_y \tilde{\psi}(y, x, \eta)\big|_{y = x}=\left.d_y \psi\left(y, \mu^{+}(x, \eta)\right)\right|_{y=x} = -\eta
\]
for complex $(y, x, \eta)$ near $\left(x_0, x_0, \eta_0\right)$.
Substituting \((z, \zeta) = \mu^+(x, \eta)\) into~\eqref{equation:bargmann}, we obtain
\begin{equation}\label{equation:bros}
\mathcal{B} \bsf(x, \eta)
:=
\int_{\mathbb{R}^n}
e^{i \lambda \tilde{\psi}(y, x, \eta)}
\left\{
W_3(x, \eta, y) + \mathcal{O}(\lambda^{-1}) \operatorname{Id}
\right\}
\bsf(y) \, dy
=
\mathcal{O}(e^{-\varepsilon \lambda}) \, \mathcal{I},
\end{equation}
uniformly for \((x, \eta)\) near $\left(x_0, \eta_0\right)$, where
\[
W_3(x, \eta, y)=
W_2(y, \mu^+(x, \eta))
=
W\left(b(w'_c(y, \mu^+(x, \eta)), y), y\right).
\]
Let \(\delta > 0\) be sufficiently small. For \(\lvert x - x_0 \rvert < \delta\), \(\lvert \eta - \eta_0 \rvert < \delta\), and \(\lambda > 1/\delta\), we have
\begin{align*}
&\left\{
W_3(x, \eta, y) + \mathcal{O}(\lambda^{-1}) \operatorname{Id}
\right\} \bsf(y)\\
&\qquad\qquad=
\left[
W_3(x_0, \eta_0, y)
+ \left(W_3(x, \eta, y) - W_3(x_0, \eta_0, y)\right)
+ \mathcal{O}(\lambda^{-1}) \operatorname{Id}
\right] \bsf(y) \\
&\qquad\qquad=
\left[
W_3(x_0, \eta_0, y)
+ \mathcal{O}(\delta) \operatorname{Id}
+ \mathcal{O}(\lambda^{-1}) \operatorname{Id}
\right] \bsf(y) \\
&\qquad\qquad=
\left(1 + \mathcal{O}(\delta) + \mathcal{O}(\lambda^{-1})\right)
W_3(x_0, \eta_0, y) \bsf(y) \\
&\qquad\qquad=
\left(1 + \mathcal{O}(\delta)\right)
W_3(x_0, \eta_0, y) \bsf(y),
\end{align*}
and thus~\eqref{equation:bros} becomes
\[
\mathcal{B} \bsf(x, \eta)
=
\int_{\mathbb{R}^n}
e^{i \lambda \tilde{\psi}(y, x, \eta)}
\left(1 + \mathcal{O}(\delta)\right)
W_3(x_0, \eta_0, y) \bsf(y) \, dy
=
\mathcal{O}(e^{-\varepsilon \lambda}) \, \mathcal{I}
\]
uniformly for \((x, \eta)\) near \((x_0, \eta_0)\).

In view of Sjöstrand's analytic microlocal analysis, the positivity of \(\operatorname{Im} \tilde{\psi}\) for real \((x, \eta)\) near \((x_0, \eta_0)\) (property~(iii) in Lemma~\ref{theorem:MST}) ensures that \(\mathcal{B}\) is a FBI transform that detects analytic wavefront sets (see Theorem~\ref{thm:FBI}). Hence, we deduce that
\[
(x_0, \eta_0) \notin \mathrm{WF}_{\mathrm{a}}\left(W_3(x_0, \eta_0, \cdot) \bsf(\cdot)\right).
\]
Since \(W_3(x_0, \eta_0, y)\) is an invertible matrix-valued function depending analytically on \(y\), we have
\[
(x_0, \eta_0) \notin \mathrm{WF}_{\mathrm{a}}(\bsf).
\]
This completes the proof of Theorem~\ref{tm:5.2}.
\end{proof}

\begin{proof}[Proof of Theorem~\ref{tm:main}]
Let $\bsf \in \mathscr{E}'(X; \mathbb{C}^N)$ be given. Choose a cutoff function $\psi \in C_c^\infty(X)$ such that $\psi = 1$ near $\hat{x}$ and $\operatorname{supp}(\psi)$ is contained in an analytic coordinate chart. We decompose $\bsf$ as
\begin{equation}\label{eq:r_2}
    \bsf = \psi \bsf + (1 - \psi)\bsf.
\end{equation}
By construction, $(1 - \psi)\bsf$ vanishes in a neighborhood of $\hat{x}$, and hence is real-analytic there. This implies
\begin{equation}\label{eq:r_1}
    (\hat{x}, \hat{\eta}) \notin \mathrm{WF}_a((1 - \psi)\bsf).
\end{equation}

According to \cite[Lemma~5.3]{mazzucchelli2023general}, the condition
\[
\left\{(z, \zeta, x, \eta) \in N^* Z \setminus 0 \mid \zeta = 0\right\} = \left\{(z, \zeta, x, \eta) \in N^* Z \setminus 0 \mid \eta = 0\right\} = \emptyset
\]
is equivalent to both projections \(\pi_{\mathcal{G}} \colon Z \to \mathcal{G}\) and \(\pi_X \colon Z \to X\) being submersions. Since \(Z\) is a double fibration, this condition is satisfied. In particular, it ensures that both \(\zeta \neq 0\) and \(\eta \neq 0\) on \(N^*Z \setminus 0\), thereby guaranteeing the assumptions of Lemma~\ref{thm:8.5.5-vector}.

Using the Bolker condition, \eqref{eq:r_1}, and Lemma~\ref{thm:8.5.5-vector}, we conclude that
\begin{equation}\label{eq:r_3}
    (\hat{z}, \hat{\zeta}) \notin \mathrm{WF}_a\big(R_W((1 - \psi)\bsf)\big).
\end{equation}

By assumption, we also have $(\hat{z}, \hat{\zeta}) \notin \mathrm{WF}_a\big(R_W(\bsf)\big).$
Combining this with the decomposition \eqref{eq:r_2} and \eqref{eq:r_3}, we obtain
\begin{equation}\label{eq:int_1}
    (\hat{z}, \hat{\zeta}) \notin \mathrm{WF}_a\big(R_W(\psi \bsf)\big).
\end{equation}

Following the argument in the proof of \cite[Theorem~5.1]{mazzucchelli2023general}, we localize near $\hat{x}$ and $\hat{z}$, and then identify $X$ and $\mathcal{G}$ with open subsets of $\mathbb{R}^n$ and $\mathbb{R}^m$, respectively. Applying Theorem~\ref{tm:5.2}, we obtain
\[
(\hat{x}, \hat{\eta}) \notin \mathrm{WF}_a(\psi \bsf).
\]
Since $\psi = 1$ in a neighborhood of $\hat{x}$, it follows that
\[
(\hat{x}, \hat{\eta}) \notin \mathrm{WF}_a(\bsf),
\]
which completes the proof.
\end{proof}
\section{Applications}
The following theorem is a straightforward generalization of the result for scalar distributions, given in \cite[Theorem~8.5.6']{hormanderiv} and \cite[Theorem 6.1]{mazzucchelli2023general}, to the case of vector-valued distributions.

\begin{theorem}\label{tm:6.1}
Let $X$ be a real-analytic manifold. Let $\Sigma$ be a $C^1$ hypersurface passing through $x_0 \in X$ with conormal $\nu_0$ at $x_0$. Suppose $\bsf=(f_1,\cdots, f_N) \in \mathcal{D}'(X, \mathbb{C}^N)$ satisfies $\bsf = 0$ on one side of $\Sigma$ near $x_0$. If for each $i\in \{1,\cdots, N\}$,
\begin{equation}\label{eq:int_3}
    (x_0, \nu_0) \notin \mathrm{WF}_a(f_i)
\quad \text{or} \quad 
(x_0, -\nu_0) \notin \mathrm{WF}_a(f_i),
\end{equation}
then $\bsf = 0$ in a neighborhood of $x_0$. In particular, if 
\begin{equation}\label{eq:vector-valued-or-statement}
   (x_0, \nu_0) \notin \mathrm{WF}_a(\bsf)
\quad \text{or} \quad 
(x_0, -\nu_0) \notin \mathrm{WF}_a(\bsf),
\end{equation}
holds, then $\bsf = 0$ in a neighborhood of $x_0$.
\end{theorem}

\begin{proof}
Since $\bsf = (f_1, \dots, f_N)$ and $\bsf = 0$ on one side of $\Sigma$ near $x_0$, it follows that each component $f_i = 0$ on one side of $\Sigma$ near $x_0$.
By applying \cite[Theorem~6.1]{mazzucchelli2023general} to each $f_i$, we conclude that $f_i = 0$ near $x_0$ for all $i$. Hence, $\bsf = 0$ near $x_0$. If \eqref{eq:vector-valued-or-statement} holds, then clearly \eqref{eq:int_3} holds for all $i \in \{1,\dots,N\}$, and the second statement follows.
\end{proof}

\begin{proof}[Proof of Theorem~\ref{tm:1.3}]
Since $R_W \bsf$ vanishes identically in a neighborhood of $z_0$, it follows that 
\[
(z_0, \zeta_0) \notin \mathrm{WF}_a(R_W \bsf).
\]
By the Bolker condition at $(z_0, \zeta_0, x_0, \nu_0)$ and Theorem~\ref{tm:main}, we conclude that
\[
(x_0, \nu_0) \notin \mathrm{WF}_a(\bsf).
\]
Since $\bsf$ vanishes on one side of a $C^1$ hypersurface $\Sigma$ near $x_0$ and $(x_0, \nu_0) \in N^* \Sigma$, Theorem~\ref{tm:6.1} implies that $\bsf = 0$ in a neighborhood of $x_0$.
\end{proof}
\begin{proposition}\label{prop:4.7}
Let $(M, g)$ be a smooth Riemannian manifold with boundary $\partial M$. Let $z_0 \in \partial_{+} S M$ satisfy $\tau_{+}(z_0) < \infty$. Assume that the geodesic $x_{z_0} \colon [0, \tau_{+}(z_0)]$ $ \to M$ satisfies the following conditions:
\begin{enumerate}
    \item $x_{z_0}$ is injective (i.e., it does not self-intersect);
    \item $x_{z_0}$ meets $\partial M$ transversely at $t = 0$ and $t = \tau_{+}(z_0)$;
    \item For all $t \in (0, \tau_{+}(z_0))$, we have $x_{z_0}(t) \in \mathrm{int}(M)$.
\end{enumerate}
Then for any sufficiently small neighborhood $\mathcal{G}$ of $z_0$ in $\partial_{+} S M$ and $W \in C^{\infty}(\mathcal{G} \times \mathrm{int}(M), \mathrm{GL}(N, \mathbb{C}))$, the matrix weighted geodesic ray transform defined by
\[
R_W \bsf(z) = \int_{0}^{\tau_{+}(z)} W(z, x_z(t)) \bsf(x_z(t)) \, \mathrm{d}t
\]
is a matrix weighted double fibration ray transform on $\mathcal{G}$. Moreover, if there are no conjugate points along $x_{z_0}$, then the Bolker condition is satisfied at every $(z_0, \zeta, x, \eta) \in C$ where $x = x_{z_0}(t)$ for some $t \in (0, \tau_{+}(z_0))$.
\end{proposition}

\begin{proof}
The proof follows exactly as in \cite[Proposition~4.7]{mazzucchelli2023general}, with the scalar weight replaced by a matrix-valued weight. The assumptions ensure that the geodesic flow defines a smooth double fibration structure, as shown in \cite[Lemma~4.3]{mazzucchelli2023general}. Furthermore, the absence of conjugate points guarantees that the Bolker condition is satisfied, as established in \cite[Lemma~4.6 and Section~4.3]{mazzucchelli2023general}.
\end{proof}
\begin{proof}[Proof of Theorem~\ref{tm:local:matrix_we}]
Since $\partial M$ is strictly convex at $x_0 \in \partial M$, it follows that one may choose a real-analytic extension $(\widetilde{M},\tilde{g})$ of $(M,g)$, with smooth boundary and $M \subset \mathrm{int}\widetilde{M}$, open $\widetilde{\mathcal{G}} \subset \partial \widetilde{M}$, and a real-analytic extension of $W$ to $C^\infty(\widetilde{\mathcal{G}} \times \widetilde{M},\mathrm{GL}(N, \mathbb{C}))$, such that $(\widetilde{M},\tilde{g})$, $\widetilde{\mathcal{G}}$, $\widetilde{W}$, and $\tilde{z}_0 \in \widetilde{\mathcal{G}}$ satisfy the assumptions of Proposition \ref{prop:4.7} and the Bolker condition holds, where $x_{\tilde{z}_0}$ is the unique geodesic in $\widetilde{M}$ such that $x_{\tilde{z}_0}$ meets $\partial M$ at $z_0=(x_0,v_0)$. Next, consider the zero extension $\tilde{\bsf}$ of $\bsf$ to $\widetilde{M}$, so that $\tilde{f} \in\mathscr{E}'(\widetilde{M},\mathbb{C}^N)$. Now, by the definition of the extension, it holds that $R_{\widetilde{W}}\tilde{\bsf}=0$ since $R_{W}{\bsf}=0$. As $\tilde{\bsf}$ vanishes on one side of the hypersurface $\partial M \subset \widetilde{M}$, Theorem~\ref{tm:1.3} implies that $\tilde{\bsf} = 0$ near $x_0$ in $\widetilde{M}$. Hence, $\bsf = 0$ near $x_0$ in $M$.
\end{proof}
\begin{proof}[Proof of Corollary~\ref{tm:1.4}]
    The statement follows from Theorem \ref{tm:local:matrix_we}, using a strictly convex foliation from \cite{BGL02}. Since the argument is virtually identical to the proof of \cite[Theorem 1.4]{mazzucchelli2023general}, we omit the details.
\end{proof}

Recall that $Z$ denotes the embedded submaniold of $\partial_+SM\times M$ defined in \eqref{def:Z}.
\begin{lemma}[Cauchy--Kovalevskaya theorem for a transport equation]\label{lm:weight_analytic}
Let \((M, g)\) be a compact real-analytic Riemannian manifold of dimension \(n \geq 2\) with real-analytic, strictly convex boundary \(\partial M\), and assume that the geodesic flow on \(M\) is nontrapping. Let $\mathcal{A}$ be a real-analytic matrix-valued function on $Z$. For $z \in \partial_+ SM$, define $W_{\mathcal{A}}$ as the solution to the first-order linear transport equation along the geodesic $x_z$
\begin{align}\label{eq:transport-equation}
\begin{split}
\frac{d}{dt} W_{\mathcal{A}}(z, x_z(t)) + \mathcal{A}(z, x_z(t)) W_{\mathcal{A}}(z, x_z(t)) &= 0, \\
W_{\mathcal{A}}(z, x_z(0)) &= \mathrm{Id}.
\end{split}
\end{align}
Then $W_{\mathcal{A}}$ is real-analytic on $Z$.
\end{lemma}
\begin{proof}
Since $\mathcal{A}$ is real-analytic on $Z$ and $M$ is a real-analytic Riemannian manifold, the geodesic $x_z(t)$ is real-analytic in $(z,t)$. Thus, \eqref{eq:transport-equation} is an initial value problem for a system of real-analytic ODEs in $t$ with real-analytic parameter $z$. Standard ODE theory implies that $W_{\mathcal{A}} \in C^\infty(Z)$. We now prove real-analyticity by establishing Cauchy estimates in local coordinates.

Let $[w_{ij}(t,z)] := W_{\mathcal{A}}(z,x_z(t))$ and $[a_{ij}(t,z)] := \mathcal{A}(z,x_z(t))$. Choose local coordinates $(z_1, \dots, z_{2n-2})$ near $z \in \partial_+SM$. In these coordinates, the system \eqref{eq:transport-equation} becomes
\begin{equation}\label{eq:transport-equation2}
\frac{\partial}{\partial t}w_{ij}(t,z) = -\sum_{k=1}^N a_{ik}(t,z)w_{kj}(t,z), \quad w_{ij}(0,z) = \delta_{ij}.
\end{equation}
We shall use the multi-index notation $\partial_z^\alpha := \partial_{z_1}^{\alpha_1} \cdots \partial_{z_{2n-2}}^{\alpha_{2n-2}}$ for partial derivatives with respect to $z$. 

Since $\mathcal A$ is real-analytic on $Z$ and the geodesic $(z,t)\mapsto x_z(t)$ is real-analytic (as $(M,g)$ is real-analytic), the function $a_{ij}(t,z) = \mathcal{A}(z, x_z(t))$ is real-analytic in $(t,z)$.
By the Cauchy--Kovalevskaya theorem (see for instance, \cite{Petrovsky:Notes}), for each $z \in \partial_+SM$, there exists a neighborhood $V \subset \partial_+SM$ and $\rho > 0$ such that
\begin{equation}\label{eq:int:Step_0}
    M := \sum_{p + |\alpha| = 0}^\infty \sum_{i,j=1}^N \frac{\rho^{p + |\alpha|}}{p! \alpha!} \left| \partial_t^p \partial_z^\alpha w_{ij}(0, z) \right| < \infty, \quad z \in V.
\end{equation}
If necessary, shrink $\rho$ and define
\begin{equation}\label{eq:int_step2}
    L := \sup_{z \in V} \sup_{t \in (0, \tau_+(z))} \sum_{p + |\alpha| = 0}^\infty \sum_{i,j=1}^N \frac{\rho^{p + |\alpha|}}{p! \alpha!} \left| \partial_t^p \partial_z^\alpha a_{ij}(t, z) \right| < \infty.
\end{equation}
Applying $\rho^{p + |\alpha|} \partial_t^p \partial_z^\alpha / (p! \alpha!)$ to \eqref{eq:transport-equation2}, we obtain the identity
\begin{align}\label{eq:int_step1}
    \begin{split}
        \frac{\partial}{\partial t} \left( \frac{\rho^{p + |\alpha|}}{p! \alpha!} \partial_t^p \partial_z^\alpha w_{ij}(t,z) \right)
&= -\sum_{k=1}^N \sum_{q \leq p} \sum_{\beta \leq \alpha}
\left( \frac{\rho^{p-q + |\alpha - \beta|}}{(p-q)! (\alpha - \beta)!} \partial_t^{p-q} \partial_z^{\alpha - \beta} a_{ik}(t,z) \right) \\
&\quad \times \left( \frac{\rho^{q + |\beta|}}{q! \beta!} \partial_t^q \partial_z^\beta w_{kj}(t,z) \right).
    \end{split}
\end{align}
This identity holds for all $t\in (0,\tau_+(z))$ since $w_{ij}$ and $(a_{ij})$ are smooth.

Define the partial sum 
\begin{align*}
    F_{\bar\nu}(t,z):=\sum_{p + |\alpha| = 0}^{\bar{\nu}} \sum_{i,j=1}^N \frac{\rho^{p + |\alpha|}}{p! \alpha!} \left| \partial_t^p \partial_z^\alpha w_{ij}(t,z) \right|.
\end{align*}
Since each $\partial_t^p \partial_z^\alpha w_{ij}(\cdot,z)$ is a smooth function of $t$, integrating \eqref{eq:int_step1} from $0$ to $t$ gives
\begin{align*}
    \frac{\rho^{p + |\alpha|}}{p! \alpha!} \partial_t^p \partial_z^\alpha w_{ij}(t,z)&=\frac{\rho^{p + |\alpha|}}{p! \alpha!} \partial_t^p \partial_z^\alpha w_{ij}(0,z)\\
    &\qquad -\int_0^t\sum_{k=1}^N \sum_{q \leq p} \sum_{\beta \leq \alpha}
\left( \frac{\rho^{p-q + |\alpha - \beta|}}{(p-q)! (\alpha - \beta)!} \partial_t^{p-q} \partial_z^{\alpha - \beta} a_{ik}(s,z) \right)\\
&\quad\quad\quad \times \left( \frac{\rho^{q + |\beta|}}{q! \beta!} \partial_t^q \partial_z^\beta w_{kj}(s,z) \right) ds.
\end{align*}
Taking absolute values and applying the triangle inequality,
\begin{align*}
    \frac{\rho^{p + |\alpha|}}{p! \alpha!} \bigl|\partial_t^p \partial_z^\alpha w_{ij}(t,z)\bigr|&\le \frac{\rho^{p + |\alpha|}}{p! \alpha!} \bigl|\partial_t^p \partial_z^\alpha w_{ij}(0,z)\bigr|\\
    &\qquad +\int_0^t\sum_{k=1}^N \sum_{q \leq p} \sum_{\beta \leq \alpha}
\left( \frac{\rho^{p-q + |\alpha - \beta|}}{(p-q)! (\alpha - \beta)!} \bigl|\partial_t^{p-q} \partial_z^{\alpha - \beta} a_{ik}(s,z)\bigr| \right)\\
&\quad\quad\quad \times \left( \frac{\rho^{q + |\beta|}}{q! \beta!} \bigl|\partial_t^q \partial_z^\beta w_{kj}(s,z) \bigr|\right) ds.
\end{align*}
Summing over $i,j\in \{1,\cdots , N\}$ and $0\le p + |\alpha|\le \bar{\nu} $, we obtain
\begin{align}\label{eq:int_step3}
    F_{\bar\nu}(t,z)\le F_{\bar\nu}(0,z)+\int_0^tS(s)ds,
\end{align}
where 
\begin{equation*}
    S(s):= \sum_{p + |\alpha| = 0}^{\bar{\nu}} \sum_{i,j,k=1}^N \sum_{q \leq p} \sum_{\beta \leq \alpha}\frac{\rho^{p-q + |\alpha - \beta|}}{(p-q)! (\alpha - \beta)!} \bigl|\partial_t^{p-q} \partial_z^{\alpha - \beta} a_{ik}(s,z)\bigr|\frac{\rho^{q + |\beta|}}{q! \beta!} \bigl|\partial_t^q \partial_z^\beta w_{kj}(s,z) \bigr|.
\end{equation*}
Substituting $r=p-q$ and $\varsigma=\alpha-\beta$ and using $p + |\alpha|\le \bar \nu$, we have $r+q+|\varsigma|+|\beta|\le \bar \nu$. This implies $r+|\varsigma|\le \bar \nu$ and $q+|\beta|\le \bar \nu$. Since all terms are nonnegative, we have
\begin{align*}
    S(s)\le \sum_{k=1}^N\left(\sum_{i=1}^N\sum_{r+|\varsigma|=0}^{\bar\nu}\frac{\rho^{r + |\varsigma|}}{r! \varsigma!} \bigl|\partial_t^{r} \partial_z^{\varsigma} a_{ik}(s,z)\bigr|\right)\left(\sum_{j=1}^N\sum_{q+|\beta|=0}^{\bar\nu}\frac{\rho^{q + |\beta|}}{q! \beta!} \bigl|\partial_t^q \partial_z^\beta w_{kj}(s,z) \bigr|\right).
\end{align*}
For each fixed $k$, the first factor is partial sum  of series in \eqref{eq:int_step2} and is therefore bounded by $L$. Therefore,
\begin{equation*}
    S(s)\le L \sum_{k,j=1}^N\sum_{q+|\beta|=0}^{\bar\nu}\frac{\rho^{q + |\beta|}}{q! \beta!} \bigl|\partial_t^q \partial_z^\beta w_{kj}(s,z) \bigr|=L F_{\bar\nu}(s,z).
\end{equation*}
Therefore from \eqref{eq:int:Step_0} and \eqref{eq:int_step3}, we have
\begin{equation}\label{eq:int:step_5}
     F_{\bar\nu}(t,z)\le M+L\int_0^tF_{\bar\nu}(s,z)ds,\qquad \qquad z\in V, t\in [0,\tau_+(z)].
\end{equation}
The function $F_{\bar\nu}(\cdot,z)$ is continuous and nonnegative on $[0,\tau_+(z)]$ and satisfies \eqref{eq:int:step_5} for all $t\in [0,\tau_+(z)]$. By the Gr\"onwall--Bellman inequality (see, for instance \cite[eq. (1.3)]{Gallegos:2025}), we have
\begin{equation}
    F_{\bar{\nu}}(t,z) \leq Me^{Lt},\qquad\qquad z\in V, t\in [0,\tau_+(z)].
\end{equation}
Since $\bar{\nu}$ is arbitrary and $M, L$ are independent of $\bar{\nu}$, the full series converges with an estimate
\[
\sum_{p + |\alpha| = 0}^\infty \sum_{i,j=1}^N \frac{\rho^{p + |\alpha|}}{p! \alpha!} \left| \partial_t^p \partial_z^\alpha w_{ij}(t,z) \right| \leq Me^{Lt},\qquad\qquad z\in V, t\in [0,\tau_+(z)].
\]
This proves that $W_{\mathcal{A}}$ is real-analytic on $Z$.
\end{proof}

We next state the local version of Lemma~\ref{lm:weight_analytic}, which may be proved similarly to Lemma~\ref{lm:weight_analytic}, but without establishing global Cauchy estimates. Thus, we omit the proof.
\begin{lemma}\label{lm:weight_analytic:local}
Let \((M, g)\) be a compact real-analytic Riemannian manifold of dimension \(n \geq 2\), and suppose that \(\partial M\) is strictly convex at a point \(x_0 \in \partial M\). Let $z_0 = (x_0, v_0) \in \partial_0 SM$, and let \(\mathcal{G}\) be a sufficiently small neighborhood of \(z_0\) in $\partial_+ SM$ such that all geodesics determined by $\mathcal{G}$ are non-trapping. Let \(\mathcal{A} \in C^\infty(\mathcal{G} \times M; \mathbb{C}^{N \times N})\) be a real-analytic matrix-valued function. For each \(z \in \mathcal{G} \setminus \partial_0 SM\), define \(W_{\mathcal{A}}\) as the solution to the first-order linear transport equation along the geodesic \(x_z\) (see \eqref{eq:transport-equation}).
Then \(W_{\mathcal{A}}\) is real-analytic on \(Z\) (see Definition~\ref{def:Z}).
\end{lemma}
\begin{remark}\label{rm:higgs_reduction}
In defining the attenuated transform of $\bsf \in C^{\infty}(M, \mathbb{C}^n)$, we consider $\pi^*_M \bsf \in C^{\infty}(Z, \mathbb{C}^n)$ in Definition~\ref{def:attenuated}. With this identification, we have $I_{\mathcal{A}} \bsf = I_{\mathcal{A}}(\pi^*_M \bsf)$.
\end{remark}
\begin{proof}[Proof of Theorem~\ref{tm:local_attenuated}]
For each $z \in \partial_+ SM$, along the geodesic $x_z$, we compute that
\begin{align*}
\frac{d}{dt} W_{\mathcal{A}}^{-1} &= -W_{\mathcal{A}}^{-1} \left( \frac{d}{dt} W_{\mathcal{A}} \right) W_{\mathcal{A}}^{-1} = W_{\mathcal{A}}^{-1} \mathcal{A}.
\end{align*}
Also, along a fixed geodesic, one obtains
\begin{align*}
\frac{d}{dt} \left( W_{\mathcal{A}}^{-1} u^{\bsf} \right)
&= \left( \frac{d}{dt} W_{\mathcal{A}}^{-1} \right) u^{\bsf} + W_{\mathcal{A}}^{-1} \frac{d}{dt} u^{\bsf} \\
&= W_{\mathcal{A}}^{-1} \mathcal{A} u^{\bsf} - W_{\mathcal{A}}^{-1} \mathcal{A} u^{\bsf} - W_{\mathcal{A}}^{-1} \pi_M^* \bsf \\
&= -W_{\mathcal{A}}^{-1} \pi_M^* \bsf.
\end{align*}
This implies that
\begin{align*}
W_{\mathcal{A}}^{-1}(z, x_z(\tau_+(z))) &u^{\bsf}(z, x_z(\tau_+(z))) - W_{\mathcal{A}}^{-1}(z, \pi_M(z)) u^{\bsf}(z, \pi_M(z)) 
\\&= -\int_0^{\tau_+(z)} W_{\mathcal{A}}^{-1}(z, x_z(t)) \bsf(x_z(t)) \, dt.
\end{align*}
Since $u^{\bsf}(z, x_z(\tau_+(z))) = 0$, one has
\begin{align*}
u^{\bsf}(z, \pi_M(z)) = W_{\mathcal{A}}(z, \pi_M(z)) \int_0^{\tau_+(z)} W_{\mathcal{A}}^{-1}(z, x_z(t)) \bsf(x_z(t)) \, dt.
\end{align*}
Using the fact that $W_{\mathcal{A}}(z, \pi_M(z))=\operatorname{Id}$, we conclude 
\begin{align}
I_{\mathcal{A}} \bsf(z) = u^{\bsf}(z, \pi_M(z)) = \int_0^{\tau_+(z)} W_{\mathcal{A}}^{-1}(z, x_z(t)) \bsf(x_z(t)) \, dt, \qquad z \in \partial_+ SM.
\end{align}
By Lemma~\ref{lm:weight_analytic:local}, if $\mathcal{A}$ is real-analytic on $\mathcal{G} \times M$, then $W_{\mathcal{A}}$ is also real-analytic on $\mathcal{G} \times M$. Thus, the attenuated ray transform $I_{\mathcal{A}} \bsf$ is of the same form as the matrix weighted ray transform $R_W \bsf$ with real-analytic weight $W = W_{\mathcal{A}}^{-1}$. The result now follows directly from Theorem~\ref{tm:local:matrix_we}.
\end{proof}
The following lemma is the pseudo-linearization identity in a local setting. The proof is identical to that of \cite[Proposition~13.2.1]{GIP2D}, and hence omitted.

\begin{lemma}[Pseudo-linearization identity]\label{lm:pseudo_identity}
Let \((M, g)\) be a compact Riemannian manifold of dimension \(n \geq 2\). Suppose that \(\partial M\) is strictly convex at a point \(x_0 \in \partial M\). Let \(z_0 = (x_0, v_0) \in \partial_0 SM\), and let \(\mathcal{G}\) be a sufficiently small neighborhood of \(z_0\) in \(\partial_+ SM\) such that all geodesics determined by $\mathcal{G}$ are non-trapping. Let \(\mathcal{A}, \mathcal{B} \in C^{\infty}(\mathcal{G} \times M, \mathbb{C}^{N \times N})\) be matrix-valued attenuation functions. If $C_{\mathcal{A}}(z) = C_{\mathcal{B}}(z)$ for all $z \in \mathcal{G}$,
then
\[
I_{E(\mathcal{A}, \mathcal{B})}(\mathcal{A} - \mathcal{B})(z) = 0 \quad \text{for all } z \in \mathcal{G},
\]
where the map \(E(\mathcal{A}, \mathcal{B}) \colon M \to \operatorname{End}(\mathbb{C}^{N \times N})\) is defined by
\[
E(\mathcal{A}, \mathcal{B})(U) := \mathcal{A} U - U \mathcal{B}.
\]
\end{lemma}
\begin{proof}[Proof of Corollary~\ref{tm:non_abelian}]
Without loss of generality, we may interpret the Higgs field \(\Phi \in C^{\infty}(M, \mathbb{C}^{N \times N})\) along a geodesic as an attenuation on \(Z\), by defining \(\Phi(x_z(t)) := \Phi(z, x_z(t))\). Since \(\Phi\) is real-analytic on \(M\), it is also real-analytic on \(Z\).

By Lemma~\ref{lm:pseudo_identity} and Remark~\ref{rm:higgs_reduction}, we obtain
\begin{equation}\label{eq:pseudo:eq}
    I_{E(\Phi, \Psi)}(\Phi - \Psi)(z) = 0, \qquad z \in \mathcal{G},
\end{equation}
where \(E(\Phi, \Psi)(U) := \Phi U - U \Psi\). Since \(\Phi\) and \(\Psi\) are real-analytic, it follows that \(E(\Phi, \Psi)\) is also real-analytic. By Lemma~\ref{lm:weight_analytic}, the associated weight \(W_{E(\Phi, \Psi)}\) is real-analytic as well. Applying Theorem~\ref{tm:local_attenuated} to the identity \eqref{eq:pseudo:eq}, we conclude that \(\Phi - \Psi = 0\) in a neighborhood of \(\pi_M(z_0)\) in $M$. Finally, by the identity theorem for real-analytic functions \cite[Corollary 1.2.5]{Krantz:Parks:2002}, it follows that \(\Phi = \Psi\) throughout \(M\).
\end{proof}
\subsection*{Acknowledgments} The work of H.C. was supported by the Grant-in-Aid for Scientific Research No.23K03186 of the Japan Society of Promotion of Science. The work of S.R.J. and J.R. was supported by the Research Council of Finland through the Flagship of Advanced Mathematics for Sensing, Imaging and Modelling (decision number 359183) and Jenny and Antti Wihuri Foundation. In addition, J.R. was supported by Emil Aaltonen Foundation and Fulbright Finland Foundation (ASLA-Fulbright Research Grant for Junior Scholars 2024--2025). We thank an anonymous referee for many helpful comments that improved the final manuscript.
\subsection*{Data availability statement} Data sharing not applicable to this article as no datasets were generated or analyzed during the current study.
\subsection*{Conflict of interest}
The authors declared no potential conflicts of interest with respect to the research, authorship, and/or publication of this article.

   \bibliography{math} 
	
	\bibliographystyle{alpha}

\end{document}